\documentclass[11pt]{amsart}

\usepackage[margin=1in]{geometry}
\usepackage{amsmath,amssymb,amsthm,mathtools}
\usepackage{enumitem}
\usepackage{microtype}
\usepackage{booktabs}
\usepackage{tikz}
\usepackage{pgfplots}
\pgfplotsset{compat=1.18}
\usepgfplotslibrary{groupplots}
\usepackage[nocompress,space]{cite}
\usepackage{xcolor}
\usepackage[
  colorlinks=true,
  linkcolor=cyan,
  citecolor=red,
  urlcolor=blue
]{hyperref}
\usepackage{aliascnt}
\newtheorem{theorem}{Theorem}[section]
\newaliascnt{proposition}{theorem}
\newtheorem{proposition}[proposition]{Proposition}
\aliascntresetthe{proposition}
\newaliascnt{lemma}{theorem}
\newtheorem{lemma}[lemma]{Lemma}
\aliascntresetthe{lemma}
\newaliascnt{corollary}{theorem}
\newtheorem{corollary}[corollary]{Corollary}
\aliascntresetthe{corollary}
\theoremstyle{definition}
\newaliascnt{definition}{theorem}

\aliascntresetthe{definition}
\newaliascnt{example}{theorem}
\newtheorem{example}[example]{Example}
\aliascntresetthe{example}
\newaliascnt{assumption}{theorem}
\newtheorem{assumption}[assumption]{Assumption}
\aliascntresetthe{assumption}
\theoremstyle{remark}
\newaliascnt{remark}{theorem}
\newtheorem{remark}[remark]{Remark}
\aliascntresetthe{remark}
\usepackage[nameinlink,capitalize,noabbrev]{cleveref}
\crefname{theorem}{Theorem}{Theorems}
\Crefname{theorem}{Theorem}{Theorems}
\crefname{proposition}{Proposition}{Propositions}
\Crefname{proposition}{Proposition}{Propositions}
\crefname{lemma}{Lemma}{Lemmas}
\Crefname{lemma}{Lemma}{Lemmas}
\crefname{corollary}{Corollary}{Corollaries}
\Crefname{corollary}{Corollary}{Corollaries}
\crefname{definition}{Definition}{Definitions}
\Crefname{definition}{Definition}{Definitions}
\crefname{example}{Example}{Examples}
\Crefname{example}{Example}{Examples}
\crefname{assumption}{Assumption}{Assumptions}
\Crefname{assumption}{Assumption}{Assumptions}
\crefname{remark}{Remark}{Remarks}
\Crefname{remark}{Remark}{Remarks}

\newcommand{\R}{\mathbb{R}}
\newcommand{\cU}{\mathcal{U}}
\newcommand{\Qcone}{\mathcal{Q}}
\newcommand{\interior}{\operatorname{int}}
\newcommand{\val}{\operatorname{val}}
\newcommand{\barr}{\operatorname{bar}_{-}}
\newcommand{\epi}{\operatorname{epi}}
\newcommand{\dist}{\operatorname{dist}}
\newcommand{\dom}{\operatorname{dom}}
\newcommand{\ran}{\operatorname{ran}}
\newcommand{\ip}[2]{\left\langle #1,#2\right\rangle}
\newcommand{\norm}[1]{\left\lVert #1\right\rVert}
\newcommand{\Id}{\mathrm{I}}
\newcommand{\cond}{\operatorname{cond}_2}
\newcommand{\trans}{\mathsf{T}}
\newcommand{\eps}{\varepsilon}

\title[Attainment Boundaries and Escape Rates]
{Attainment Boundaries and Escape Rates in Asymptotically Conic Optimization}

\author{Vinh Nguyen}
\address{Department of Mathematics, University of California, Berkeley, CA 94720, USA}
\email{vnguyen26@berkeley.edu}

\date{}
\begin{document}
\begin{abstract}
We study attainment boundaries for linear optimization over unbounded convex sets.  For epigraphs of finite convex functions, convex conjugacy separates recession-cone copositivity, boundedness below, and attainment through three nested subsets of conjugate space.  At a finite but unattained boundary value, we establish a facewise asymptotic selection theorem for inward perturbations of the objective.  The theorem identifies the escaping directions of the perturbed minimizers, their precise blow-up scale, and the leading asymptotics of the optimal value; the critical set may be multidimensional, and no radial symmetry is assumed.  For radial asymptotically conic epigraphs, we obtain a complete boundary trichotomy and universal scaling laws for objective tilting, hard truncation, and power regularization.  For shifted ellipsoidal second-order cone programs, we derive explicit primal--dual formulas together with sharp escape, conditioning, and regularization rates.  These models also admit an exact robust-optimization representation.  At the attainment boundary, strict primal feasibility, zero duality gap, and dual attainment can coexist with failure of primal attainment.
\end{abstract}

\maketitle
\section{Introduction}\label{sec:introduction}

A linear functional that is bounded below on a nonempty polyhedron attains its minimum.  This elementary-looking principle is one of the structural reasons why linear programming has such a stable solvability theory; classical formulations can be found in \cite{R1970,BS2000}.  In recession-cone form, if $P$ is a nonempty polyhedron and $c$ is nonnegative on $P^\infty$, then the linear program $\inf_{x\in P}\ip{c}{x}$ has a finite value and a solution.  The statement is usually treated as a single theorem, but it combines three different properties: nonnegativity on recession directions, boundedness of the objective, and attainment of the infimum.  Outside polyhedral optimization these properties separate, as already illustrated by early work on convex programs with unattained infima \cite{A1975}.

The separation matters in modern conic models.  Second-order and semidefinite cones generate tractable formulations for robust optimization, control, statistics, and machine learning \cite{AG2003}, but their affine slices and linear images need not inherit the closedness properties of polyhedra.  Closed linear images are central in conic duality \cite{P2007}.  Recent work develops solution-existence criteria under weak regularity, constraint-qualification-free optimality conditions, and sharp error bounds for nonpolyhedral cones \cite{AHMRS2023,KT2024,LY2024,LLP2025,LLLP2025}.  For nonsmooth problems with unbounded feasible regions, Hung, Chuong, and Anh \cite{HCA2025} characterize finiteness, existence, and compactness of global solution sets through coercivity, sublevel compactness, and asymptotic geometry.  Facial reduction is a standard geometric tool for exposing the minimal face responsible for degeneracy \cite{BW1981}; recent extensions carry this idea from cones to general convex sets and convex functions \cite{LLL2025}.  Primal--dual level-set geometry also relates the size of primal level sets to dual interiority \cite{F2003}.

The problem considered here draws on three classical themes.  Convex conjugacy and recession analysis describe boundedness and asymptotic slopes \cite{R1970,RW1998}; conic duality and facial reduction explain how degeneracy affects dual representations and constraint qualifications \cite{BW1981,P2007}; and condition theory quantifies the deterioration of optimization problems as the data approach ill-posedness \cite{R1995,PR2020}.  Here the feasible set remains fixed while the objective approaches a slope for which the optimal value stays finite but primal attainment disappears.  Nearly optimal primal points then escape to infinity even though feasibility and duality may remain regular.  We identify this boundary in conjugate space and quantify the escape directions, blow-up rates, value asymptotics, and loss of conditioning.

Recent work on conic optimization treats several nearby issues.  Luan and Yen \cite{LY2024} obtain strong duality and primal--dual solution existence under weak separation and generalized Slater-type hypotheses, while Lin et al. \cite{LLLP2025} derive tight error bounds for the log-determinant cone without imposing constraint qualifications.  Lin, Liu, and Louren\c{c}o \cite{LLL2025} further extend facial-reduction ideas from cones to general convex sets and convex functions.  These works concern feasibility, constraint qualification, duality, or distance to a feasible set.  In the setting studied here, the feasible epigraph is fixed and may satisfy Slater's condition; only the objective slope moves toward a boundary point where the infimum is finite but unattained.  Since no finite optimizer exists at the limit, the asymptotic object is a face of directions at infinity.  In the ellipsoidal SOCP, strict primal feasibility and dual attainment persist at the nonattainment boundary.

We study attainment boundaries and escape rates for unbounded convex sets that approach cones at infinity.  We begin with the epigraph $E_f:=\{(t,z)\in\R\times X:t\geq f(z)\}$ of a finite convex function on a finite-dimensional space.  For a linear objective $\ell_{a,b}(t,z)=at+b(z)$ with $a>0$, set $p:=-b/a$.  The three solvability properties are encoded by the nested sets
\[
  \ran\partial f\subseteq\dom f^*\subseteq\overline{\dom f^*}.
\]
The objective is copositive on $E_f^\infty$ exactly when $p\in\overline{\dom f^*}$, it is bounded below exactly when $p\in\dom f^*$, and it attains its minimum exactly when $p\in\ran\partial f$.  Thus the set $\overline{\dom f^*}\setminus\dom f^*$ describes objectives that are copositive but unbounded below, whereas $\dom f^*\setminus\ran\partial f$ describes objectives with finite unattained infima.  This classification, established in \Cref{thm:conjugate-classification}, provides a conjugate-geometric refinement of recession-cone analysis.  It also recovers the polyhedral theorem because a finite polyhedral convex function satisfies $\ran\partial f=\dom f^*=\overline{\dom f^*}$.

The same classification also controls minimizing sequences.  When $p\in\dom f^*\setminus\ran\partial f$, every minimizing sequence escapes to infinity.  Every cluster point of its normalized directions belongs to the normal cone of $\overline{\dom f^*}$ at $p$, or equivalently generates a zero-cost recession direction.  Hence finite nonattainment can occur only along recession directions normal to the boundary slope in conjugate space.

The main quantitative result describes directional selection and escape rates near a finite nonattainment boundary.  At a nonattained boundary slope $p_0$, write $g(x)=f(x)-p_0(x)+f^*(p_0)$ and suppose that, uniformly in direction, $g(ru)$ admits a first recession term plus a lower-order positive defect of order $r^{-\nu}$.  For the inward perturbation $p_\mu=p_0-\mu h$, \Cref{thm:facewise-selection} shows that every optimizer escapes at the scale $\mu^{-1/(\nu+1)}$, while its normalized direction is selected from the zero-recession face by the directional defect coefficient and the inward tilt.  The critical face may be multidimensional, and radial symmetry is not assumed.  \Cref{ex:nonradial-multidimensional-face} gives an explicit nonradial convex function whose critical set is a two-dimensional spherical patch and for which an inward tilt selects a unique escaping direction.

The selection theorem differs from the optimal-path analysis of Auslender, Cominetti, and Haddou \cite{ACH1997}.  Their framework studies primal and dual trajectories generated by broad classes of penalty and barrier functions and establishes convergence of those trajectories to primal and dual optimal sets; in linear programming the limiting points are singletons.  Here the unperturbed problem has a finite infimum but no primal optimizer, so the perturbed minimizers cannot converge to a finite primal optimal set.  The perturbation here is an inward change of the objective slope, and the conclusion describes escape to infinity: the selected directions, the blow-up scale, and the leading value law, including the case of a multidimensional critical face.  The penalty/barrier framework of \cite{ACH1997}, in turn, covers algorithmically generated optimal paths under assumptions that are not tied to the facewise recession expansion used here.

The radial case is described by the epigraphs
\[
  C_{\phi,\rho}:=
  \{(t,z)\in\R\times X:t\geq\phi(\rho(z))\},
\]
where $\rho$ is a norm and $\phi$ is finite, convex, nondecreasing, and asymptotically linear with slope $L>0$.  The recession cone is
\[
  K_{L,\rho}:=\{(s,w):s\geq L\rho(w)\}.
\]
The conjugate formula reduces to one scalar variable, and the distinction among the three objective regions is determined by the asymptotic defect
\[
  \psi(r):=\phi(r)-Lr.
\]
Convexity forces $\psi$ to be nonincreasing.  At a nonzero boundary objective, the limit of $\psi$ produces an exact trichotomy.  If $\psi(r)\to-\infty$, the objective is copositive on the recession cone but unbounded below.  If $\psi$ reaches a finite limiting value, then $\phi$ is eventually affine and the minimum is attained.  If $\psi$ approaches a finite limit without reaching it, then the infimum is finite and unattained.  These alternatives have the same recession cone, so they isolate precisely the information that recession geometry omits.

Assume in addition that the boundary defect satisfies
\[
  \psi(r)-\gamma\sim\kappa r^{-\nu},
  \qquad \kappa>0,\quad \nu>0.
\]
A convex regular-variation lemma gives the corresponding asymptotics for the one-sided derivatives.  The same exponent $\nu$ then governs three standard approximation mechanisms.  An objective tilt by a copositivity margin $\mu$ selects decisions of size $\mu^{-1/(\nu+1)}$ and produces an optimal-value excess of order $\mu^{\nu/(\nu+1)}$.  A hard decision bound that allows error $\eta$ must grow like $\eta^{-1/\nu}$.  A power regularizer with coefficient $\lambda$ and exponent $q>1$ selects decisions of size $\lambda^{-1/(\nu+q)}$ and residual of order $\lambda^{\nu/(\nu+q)}$.  The constants are explicit; see \Cref{sec:universal-rates}.  These rates are related to the condition-number principle that deterioration is governed by proximity to a boundary of well-posed instances \cite{R1995,PR2020}, and to recent facial-residual analyses of non-Lipschitz error scales for nonpolyhedral and perspective cones \cite{LLP2025,LLLP2025,WLP2025}.  For example, \cite{WLP2025} studies the frequent one-half H\"older exponent in conic error bounds through facial reduction.  The quantities studied here are instead the size and local conditioning of primal optimizers as the objective approaches the attainment boundary.

A central example is the shifted ellipsoidal epigraph
\[
  C_{\delta,H}:=
  \left\{(t,z)\in\R\times\R^n:
  t\geq\sqrt{\delta^2+z^{\trans}Hz}\right\},
  \qquad H\succ0,
\]
which is second-order-cone representable.  Let $\beta_H(b)=\sqrt{b^{\trans}H^{-1}b}$.  Minimizing $at+b^{\trans}z$ over $C_{\delta,H}$ gives value $-\infty$ when $a<\beta_H(b)$, a unique minimizer and value $\delta\sqrt{a^2-\beta_H(b)^2}$ when $a>\beta_H(b)$, and finite value zero without a minimizer when $a=\beta_H(b)>0$.  At the boundary, the primal problem is strictly feasible, its dual optimum is attained, and the duality gap is zero.  Only primal attainment fails.  Exact formulas show that interior minimizers diverge at the inverse square-root rate, the reduced Hessian condition number diverges at the inverse-margin rate, and boundary $\eta$-optimal solutions require norm of order $\eta^{-1}$.  Quadratic Tikhonov regularization with coefficient $\eps$ yields the distinct scale $\eps^{-1/3}$, in agreement with the general power-law theorem for $\nu=1$ and $q=2$.

The same ellipsoidal set is the exact robust counterpart of a family of affine inequalities with joint ellipsoidal uncertainty in the intercept and slope coefficients.  Robust optimization with ellipsoidal uncertainty is a standard source of second-order cone models \cite{BN1998,BEN2009}.  Recent work analyzes robust feasibility radii, optimality and duality, and stability for uncertain conic systems \cite{GJLV2022}.  Related papers by Chuong and collaborators study robust optimality and duality, adjustable Farkas-type alternatives, stable conic relaxations, and conic reformulations for uncertain convex and multiobjective models \cite{C2020,CJ2020,CV2023,HCA2024,CYELC2025}.  Those papers focus on certification, duality, and reformulation under uncertainty.  In the model studied here, feasibility and boundedness persist, and the conic reformulation satisfies Slater's condition with zero duality gap, yet the infimum may fail to be attained.  This occurs on a codimension-one cost boundary, with explicit blow-up of nearby optimizers.

The paper is organized as follows.  \Cref{sec:preliminaries} collects the convex-analytic facts used throughout.  \Cref{sec:conjugate} proves the general epigraph classification and the directional theorem for minimizing sequences.  \Cref{sec:facewise} establishes the nonradial facewise selection theorem and the associated escape-rate law for inward objective perturbations.  \Cref{sec:radial} gives the radial boundary trichotomy, and \Cref{sec:universal-rates} establishes the regular-variation scaling laws for tilting, truncation, and regularization.  \Cref{sec:socp-models} develops the exact ellipsoidal SOCP phase diagram and its primal--dual formulas.  \Cref{sec:quantitative} derives sharp sensitivity, conditioning, and Tikhonov estimates.  Finally, \Cref{sec:robust-extensions} gives the robust-optimization interpretation.

\section{Convex-analytic preliminaries}\label{sec:preliminaries}

Throughout, all vector spaces are finite-dimensional and real.  For a set $A$, we denote its closure by $\bar A$.  Let $X$ be such a space, let $X^*$ be its algebraic dual, and let $\rho$ be a norm on $X$.  The zero-dimensional case is immediate, so we assume that $X\neq\{0\}$.  The dual norm is
\[
  \rho_*(b):=\sup\{b(z):\rho(z)\leq1\},
  \qquad b\in X^*.
\]
Since the unit ball is compact and symmetric, this agrees with the usual dual norm and the supremum is attained.  For every $b\in X^*$, define the set of norming directions
\[
  \mathcal D_\rho(b):=
  \{u\in X:\rho(u)=1,\ b(u)=\rho_*(b)\}.
\]
It is nonempty and compact; when $b=0$, it is the entire unit sphere.  Symmetry of the norm gives
\[
  \min_{\rho(z)=r}b(z)=-r\rho_*(b),
  \qquad r\geq0,
\]
with minimizers $z=-ru$, $u\in\mathcal D_\rho(b)$ when $r>0$.

For a nonempty closed convex set $C\subset Y$, its recession cone is
\[
  C^\infty:=\{d\in Y:x+\lambda d\in C
  \text{ for every }x\in C\text{ and every }\lambda\geq0\}.
\]
A linear functional $c\in Y^*$ is \emph{copositive} on $C^\infty$ if $c(d)\geq0$ for all $d\in C^\infty$, and \emph{strictly copositive} if $c(d)>0$ for every nonzero $d\in C^\infty$.

We use the lower barrier cone
\[
  \barr(C):=\left\{c\in Y^*:
  \inf_{x\in C}c(x)>-\infty\right\}.
\]
Standard convex analysis gives
\begin{equation}\label{eq:barrier-closure}
  \overline{\barr(C)}=(C^\infty)^*,
\end{equation}
where the dual cone is taken with the nonnegative pairing convention; see \cite[Sections~8 and 14]{R1970}, and \cite{RW1998}.  The possible failure of $\barr(C)$ to be closed is one reason why copositivity on $C^\infty$ need not imply boundedness below.

Strict copositivity also yields level boundedness, a fact used repeatedly below.
\begin{proposition}\label{prop:strict-copos}
Let $C\subset Y$ be nonempty, closed, and convex, and let $c\in Y^*$.  If $c$ is strictly copositive on $C^\infty$, then every sublevel set $\{x\in C:c(x)\leq\alpha\}$ is bounded.  Consequently, $c$ has a finite minimum on $C$.
\end{proposition}

\begin{proof}
Fix any norm on $Y$.  Assume that a sublevel set is unbounded.  Choose $x_k\in C$ and $\bar x\in C$ such that
\[
  c(x_k)\leq\alpha,
  \qquad r_k:=\norm{x_k-\bar x}\to\infty.
\]
After passing to a subsequence, $(x_k-\bar x)/r_k\to d$ with $\norm d=1$.  The standard recession-direction argument for closed convex sets gives $d\in C^\infty$, while
\[
  c(d)=\lim_{k\to\infty}\frac{c(x_k)-c(\bar x)}{r_k}\leq0,
\]
contradicting strict copositivity.  Thus every sublevel set is bounded.  The infimum cannot be $-\infty$, since then a sequence with $c(x_k)\to-\infty$ would eventually lie in a fixed bounded sublevel set.  Hence a minimizing sequence lies in a compact sublevel set, and continuity of $c$ yields a minimizer.
\end{proof}

\begin{remark}\label{rem:strict-boundary}
Objectives in the interior of $(C^\infty)^*$ therefore have bounded sublevel sets and attain their minima.  On the boundary, the objective vanishes on at least one nonzero recession direction, and attainment may fail.
\end{remark}

\section{Epigraph solvability through convex conjugacy}\label{sec:conjugate}

For finite convex epigraphs, the recession formula, the support-function representation of $f^\infty$, Fenchel duality, and subgradient inversion are classical; see, for example, \cite{R1970,RW1998}.  \Cref{thm:conjugate-classification} combines these facts into a three-level classification in slope space and identifies the two gaps between the resulting regions.  Part~\textup{(v)} determines whether the endpoint of the objective image is attained.

Let $f:X\to\R$ be finite and convex.  It is therefore continuous.  Its Fenchel conjugate is
\[
  f^*(p):=\sup_{z\in X}\{p(z)-f(z)\},
  \qquad p\in X^*,
\]
and its recession function is
\begin{equation}\label{eq:recession-function}
  f^\infty(w):=
  \lim_{\lambda\to\infty}
  \frac{f(z+\lambda w)-f(z)}{\lambda}.
\end{equation}
The limit exists in $\R\cup\{+\infty\}$ and is independent of $z$.  Standard conjugacy gives
\begin{equation}\label{eq:recession-support}
  f^\infty=\sigma_{\dom f^*}
  =\sigma_{\overline{\dom f^*}},
\end{equation}
where $\sigma_D(w)=\sup_{p\in D}p(w)$ is the support function; see \cite[Theorems~8.5 and 13.3]{R1970}.

For the epigraph $E_f:=\epi f=\{(t,z):t\geq f(z)\}$, consider the objective $\ell_{a,b}(t,z):=at+b(z)$.  Theorem~\ref{thm:conjugate-classification} identifies the boundary between copositivity, boundedness below, and solvability.

\begin{theorem}\label{thm:conjugate-classification}
Let $f:X\to\R$ be finite and convex.  If $a<0$, then $\ell_{a,b}$ is unbounded below on $E_f$; the same holds when $a=0$ and $b\neq0$, while the zero objective is attained everywhere.  Assume henceforth that $a>0$ and set $p:=-b/a$.  Then:

\begin{enumerate}[label=\textup{(\roman*)}]
\item The recession cone is
\begin{equation}\label{eq:epi-recession}
  E_f^\infty=\epi f^\infty.
\end{equation}
\item The objective is copositive on $E_f^\infty$ if and only if $p\in\overline{\dom f^*}$.
\item The objective is bounded below on $E_f$ if and only if $p\in\dom f^*$; in that case, 
\[
\inf_{E_f}\ell_{a,b}=-a f^*(p).
\]
\item The infimum is attained if and only if $p\in\ran\partial f=\dom\partial f^*$, and then the minimizers are exactly 
\[
\big\{(f(z),z):p\in\partial f(z)\big\}.
\]
\item If $p\in\dom f^*\setminus\ran\partial f$, then $\ell_{a,b}(E_f)=(-a f^*(p),\infty)$; if $p\in\ran\partial f$, the left endpoint is included.
\end{enumerate}
\end{theorem}

\begin{proof}
The cases $a\leq0$ follow as in linear programming: when $a<0$, increase the epigraph variable $t$; when $a=0$ and $b\neq0$, move $z$ along a direction on which $b$ is negative.

Formula \eqref{eq:epi-recession} is the standard epigraph identity for recession functions.  It can also be verified directly from \eqref{eq:recession-function}.  For $a>0$, copositivity on $\epi f^\infty$ is equivalent to
\[
  a f^\infty(w)+b(w)\geq0
  \qquad(w\in X), \quad \text{ or}
\]
\[
  p(w)\leq f^\infty(w)
  =\sigma_{\overline{\dom f^*}}(w)
  \qquad(w\in X).
\]
By the separation theorem, the latter condition is equivalent to $p\in\overline{\dom f^*}$.

Eliminating $t$ gives
\begin{align*}
  \inf_{(t,z)\in E_f}\ell_{a,b}(t,z)
  &=a\inf_{z\in X}\{f(z)-p(z)\}
  =-a f^*(p).
\end{align*}
This value is finite exactly on $\dom f^*$.  Equality in Fenchel's inequality
\[
  f(z)+f^*(p)\geq p(z)
\]
holds exactly when $p\in\partial f(z)$, equivalently $z\in\partial f^*(p)$.  This proves the attainment criterion and the minimizer formula.

Finally, let $m$ denote the infimum.  For every $v>m$, the definition of the infimum provides a feasible point whose objective value is strictly below $v$; increasing its $t$-coordinate then realizes the value $v$.  Hence the image is the upper half-line based at $m$, with the endpoint included exactly when the infimum is attained.
\end{proof}

The theorem produces three intrinsic regions in the normalized slope space $X^*$: $\mathcal S_f:=\ran\partial f$, $\mathcal B_f:=\dom f^*$, and $\mathcal C_f:=\overline{\dom f^*}$.  For an objective with $a>0$, membership of the normalized slope $p=-b/a$ in these sets means, respectively, solvability, boundedness below, and recession-cone copositivity.

\begin{corollary}\label{cor:three-regions}
For every finite convex $f$,
\[
  \mathcal S_f\subseteq\mathcal B_f\subseteq\mathcal C_f,
  \qquad
  \overline{\mathcal B_f}=\mathcal C_f.
\]
The set $\mathcal C_f\setminus\mathcal B_f$ consists exactly of normalized slopes of positive-$t$ objectives that are copositive but unbounded below.  The set $\mathcal B_f\setminus\mathcal S_f$ consists exactly of normalized slopes of objectives with finite unattained infima.
\end{corollary}

\begin{proof}
The inclusions and interpretations follow from \Cref{thm:conjugate-classification}.  The closure identity is the definition of $\mathcal C_f$.
\end{proof}

For a finite polyhedral convex function, the three sets coincide, recovering the classical polyhedral theorem.

\begin{corollary}\label{cor:polyhedral-recovery}
If $f$ is finite and polyhedral convex, then
\begin{equation}\label{eq:polyhedral-equality}
  \ran\partial f=\dom f^*=\overline{\dom f^*}.
\end{equation}
Consequently, every objective copositive on $E_f^\infty$ is bounded below and attains its minimum on $E_f$.
\end{corollary}

\begin{proof}
A finite polyhedral convex function is the maximum of finitely many affine functions.  Its conjugate has a compact polyhedral domain, and a proper polyhedral convex function is subdifferentiable at every point of its domain; see \cite[Sections~19 and 23]{R1970}.  Hence $\dom\partial f^*=\dom f^*$, and \eqref{eq:polyhedral-equality} follows from conjugate subgradient inversion.
\end{proof}

Nonattainment also has a directional signature in primal space.  Fix any norm $\norm{\cdot}_0$ on $X$.  Recall that $\mathcal C_f=\overline{\dom f^*}$.  For $p\in\mathcal C_f$, its normal cone is
\[
  N_{\mathcal C_f}(p):=
  \{w\in X:\widetilde p(w)\leq p(w)\text{ for every }\widetilde p\in\mathcal C_f\}.
\]

\begin{theorem}\label{thm:minimizing-directions}
Assume $p\in\dom f^*\setminus\ran\partial f$ and let $(z_k)$ satisfy
\begin{equation}\label{eq:min-sequence-f}
  f(z_k)-p(z_k)\longrightarrow-f^*(p).
\end{equation}
Then $\norm{z_k}_0\to\infty$.  Every cluster point $u$ of $z_k/\norm{z_k}_0$ satisfies $f^\infty(u)=p(u)$ and $u\in N_{\mathcal C_f}(p)$.  Thus $(f^\infty(u),u)$ is a nonzero recession direction on which the corresponding epigraph objective vanishes.
\end{theorem}

\begin{proof}
If $(z_k)$ had a bounded subsequence, continuity of $f-p$ would produce a minimizer, contradicting $p\notin\ran\partial f$.  Hence $\norm{z_k}_0\to\infty$.  Pass to a subsequence such that
\[
  u_k:=\frac{z_k}{\norm{z_k}_0}\longrightarrow u,
  \qquad \norm{u}_0=1.
\]
For every $\widetilde p\in\dom f^*$, Fenchel's inequality gives
\[
  f(z_k)\geq \widetilde p(z_k)-f^*(\widetilde p).
\]
After division by $\norm{z_k}_0$ and passage to the limit,
\[
  \liminf_{k\to\infty}
  \frac{f(z_k)}{\norm{z_k}_0}
  \geq \widetilde p(u).
\]
Taking the supremum over $\widetilde p$ and using \eqref{eq:recession-support} yields a lower bound by $f^\infty(u)$.  On the other hand, \eqref{eq:min-sequence-f} implies
\[
  \frac{f(z_k)}{\norm{z_k}_0}-p(u_k)\longrightarrow0,
\]
so $p(u)\geq f^\infty(u)$.  Since $p\in\mathcal C_f$, the support-function inequality gives $p(u)\leq f^\infty(u)$.  Equality follows.  The identity $f^\infty(u)=\sigma_{\mathcal C_f}(u)=p(u)$ is equivalent to $u\in N_{\mathcal C_f}(p)$.
\end{proof}

\begin{remark}\label{rem:directional-converse}
The theorem gives a necessary asymptotic direction but not a converse for arbitrary convex functions.  A normal direction of $\mathcal C_f$ may fail to generate a minimizing sequence because lower-order geometry can obstruct approach to the conjugate value.  In the radial class below, symmetry and the scalar defect remove this ambiguity and produce canonical minimizing sequences in every norming direction.
\end{remark}

\section{Facewise asymptotic selection under inward tilting}\label{sec:facewise}

When the critical normal cone contains more than one ray, the preceding theorem does not determine which escape direction is selected.  A uniform asymptotic expansion at infinity resolves this question for inward perturbations of the objective.

Fix a norm $\norm{\cdot}_0$ on $X$ and let 
\[
S_0:=\{u\in X:\norm{u}_0=1\}.
\]
Let $f:X\to\R$ be finite and convex, choose $p_0\in\dom f^*\setminus\ran\partial f$, and define the nonnegative boundary gap 
\[
g(x):=f(x)-p_0(x)+f^*(p_0). 
\]
Then $\inf_X g=0$ and the infimum is not attained.  On $S_0$ set 
\[
\Psi(u):=f^\infty(u)-p_0(u)\geq0
\]
and define the critical face 
\[
\cU:=\{u\in S_0:\Psi(u)=0\}.
\]
Since $\mathcal C_f=\overline{\dom f^*}$ and $f^\infty=\sigma_{\mathcal C_f}$, we have $\cU=N_{\mathcal C_f}(p_0)\cap S_0$.  Thus $\cU$ is the normalized set of critical normal directions at $p_0$, and by \Cref{thm:minimizing-directions} every cluster direction of a minimizing sequence belongs to $\cU$.

To capture the first nonvanishing correction to the recession geometry, we impose a uniform expansion in the direction variable; the critical set $\cU$ may therefore have positive dimension.

\begin{assumption}\label{ass:facewise-expansion}
The function $\Psi$ is finite and continuous on $S_0$, and the set $\cU$ is nonempty.  There exist $\nu>0$ and a continuous function $\kappa:S_0\to(0,\infty)$ such that
\begin{equation}\label{eq:uniform-facewise-expansion}
  \sup_{u\in S_0}
  \left|
    r^\nu\bigl(g(ru)-r\Psi(u)\bigr)-\kappa(u)
  \right|
  \longrightarrow0
  \qquad(r\to\infty).
\end{equation}
\end{assumption}

\paragraph{Sufficient conditions for \Cref{ass:facewise-expansion}.}
Assumption~\ref{ass:facewise-expansion} can be checked directly in several common classes.  We record two sufficient conditions that cover the examples below.

First, consider a finite sum of asymptotically linear norm profiles
\begin{equation}\label{eq:profile-sum-template}
  f(x)=\ell(x)+c+\sum_{i=1}^m \phi_i(\rho_i(x)),
\end{equation}
where $\ell\in X^*$, each $\rho_i$ is a norm on $X$, each $\phi_i:[0,\infty)\to\R$ is finite, convex, and nondecreasing, and, for a common exponent $\nu>0$,
\begin{equation}\label{eq:profile-sum-asymptotic}
  \phi_i(s)=L_i s+\gamma_i+k_i s^{-\nu}+o(s^{-\nu})
  \qquad(s\to\infty),
\end{equation}
with $L_i\geq0$, $k_i\geq0$, and not all $k_i$ zero.  Let $p_0\in\dom f^*\setminus\ran\partial f$ satisfy
\begin{equation}\label{eq:profile-sum-normalization}
  c+\sum_{i=1}^m\gamma_i+f^*(p_0)=0.
\end{equation}
Then norm equivalence on the compact sphere $S_0$ gives constants $0<m_i\leq \rho_i(u)\leq M_i$ for $u\in S_0$.  Substituting $s=r\rho_i(u)$ into \eqref{eq:profile-sum-asymptotic} therefore makes every remainder uniform in $u$, and \Cref{ass:facewise-expansion} follows with
\begin{equation}\label{eq:profile-sum-Psi-kappa}
  \Psi(u)=\ell(u)+\sum_{i=1}^mL_i\rho_i(u)-p_0(u),
  \qquad
  \kappa(u)=\sum_{i=1}^m k_i\rho_i(u)^{-\nu}.
\end{equation}
Terms with faster decay can be absorbed into the remainder.  This template includes \Cref{ex:nonradial-multidimensional-face} and, with one norm, the ellipsoidal model in \Cref{cor:ellipsoid-facewise}.

A second criterion applies when $X=\R^n$ with its Euclidean pairing and $f$ is smooth along large rays.  Suppose that $f$ is $C^1$ outside a ball, $f^\infty$ is finite on $S_0$, and for some continuous $\kappa:S_0\to(0,\infty)$,
\begin{align}
  f(ru)-r f^\infty(u)&\longrightarrow -f^*(p_0)
  &&\text{uniformly for }u\in S_0,\label{eq:smooth-facewise-constant}\\
  \ip{\nabla f(ru)}{u}-f^\infty(u)
  &=-\nu\kappa(u)r^{-\nu-1}+o(r^{-\nu-1})
  &&\text{uniformly for }u\in S_0.\label{eq:smooth-facewise-derivative}
\end{align}
Integrating \eqref{eq:smooth-facewise-derivative} from $r$ to infinity and using \eqref{eq:smooth-facewise-constant} gives
\[
  f(ru)-r f^\infty(u)+f^*(p_0)
  =\kappa(u)r^{-\nu}+o(r^{-\nu})
\]
uniformly on $S_0$, which is exactly \eqref{eq:uniform-facewise-expansion}.  A $C^2$ variant replaces \eqref{eq:smooth-facewise-derivative} by the uniform expansion
$\ip{\nabla^2 f(ru)u}{u}=\nu(\nu+1)\kappa(u)r^{-\nu-2}+o(r^{-\nu-2})$, together with the limiting radial slope and \eqref{eq:smooth-facewise-constant}; integrating twice gives the same conclusion.

The set $\cU$ has an analogy with active-manifold identification and partial smoothness, although the setting is different.  Partial smoothness is a local theory around a finite critical point: under suitable regularity, nearby critical points or algorithmic iterates identify a smooth active manifold \cite{L2002,LW2011}.  Here the boundary slope $p_0$ has no primal critical point at all.  The relevant ``active set'' is instead the normal face at infinity $N_{\mathcal C_f}(p_0)\cap S_0$, and \Cref{ass:facewise-expansion} supplies a uniform higher-order model on that set.  Thus \Cref{thm:facewise-selection} gives asymptotic directional identification rather than finite identification of a manifold.  On a smooth stratum of $\cU$, partial smoothness may provide a finer local analysis after a direction has been selected.

Assumption~\ref{ass:facewise-expansion} is also different from a constraint qualification or a facial-reduction hypothesis.  Facial reduction changes the representation of a problem to recover regularity or strong duality; see, for example, \cite{LLL2025}.  Here no such repair is needed: small inward tilts are coercive, and the SOCP model satisfies Slater's condition even at the limiting boundary.  The failure at $p_0$ is instead the absence of a finite primal optimizer.  The assumption quantifies how the objective gap approaches its recession face, and the coefficient $\kappa$ and exponent $\nu$ then determine the escape rate.
\medskip

The assumption allows multidimensional critical faces for functions that are not radial of the form $\phi\circ\rho$.

\begin{example}[A nonradial model with a two-dimensional critical face]\label{ex:nonradial-multidimensional-face}
Let $X=\R^3$, equip $X$ with the Euclidean reference norm $\norm{\cdot}_0=\norm{\cdot}_2$, and define two polyhedral norms by
\[
  \rho_1(x):=\max\{|x_1|,|x_2|,|x_3|\},
  \qquad
  \rho_2(x):=\max\{|x_1|,2|x_2|,3|x_3|\}.
\]
Consider the finite convex function
\begin{equation}\label{eq:nonradial-face-example-f}
  f(x):=\sqrt{1+\rho_1(x)^2}+\sqrt{1+\rho_2(x)^2},
  \qquad x\in\R^3,
\end{equation}
and the boundary slope $p_0(x):=2x_1$.
Then $p_0\in\dom f^*\setminus\ran\partial f$, one has $f^*(p_0)=0$, and \Cref{ass:facewise-expansion} holds with $\nu=1$ and
\begin{equation}\label{eq:nonradial-face-example-kappa}
  \kappa(u)
  =\frac{1}{2\rho_1(u)}+\frac{1}{2\rho_2(u)},
  \qquad u\in S_0=S^2.
\end{equation}
Moreover, the critical face is the two-dimensional spherical patch
\begin{equation}\label{eq:nonradial-face-example-U}
  \cU
  =\left\{
    u\in S^2:
    u_1\geq 2|u_2|,\quad
    u_1\geq 3|u_3|
  \right\}.
\end{equation}
In particular, $\cU$ contains a nonempty relatively open subset of $S^2$ and is not a single direction.

Indeed, since $\sqrt{1+s^2}>s$ for every $s\geq0$ and $\rho_i(x)\geq|x_1|$, we have
\[
  f(x)-p_0(x)>0
  \qquad(x\in\R^3).
\]
On the other hand, along the ray $x=re_1$,
\[
  f(re_1)-p_0(re_1)
  =2\bigl(\sqrt{1+r^2}-r\bigr)
  \longrightarrow0.
\]
Consequently,
\[
  f^*(p_0)
  =\sup_{x\in\R^3}\{p_0(x)-f(x)\}=0,
\]
but the supremum is not attained.  By Fenchel equality this is equivalent to
$p_0\notin\ran\partial f$.

The recession function is
\[
  f^\infty(u)=\rho_1(u)+\rho_2(u),
\]
so
\[
  \Psi(u)=\rho_1(u)+\rho_2(u)-2u_1.
\]
Because $\rho_i(u)\geq |u_1|\geq u_1$, equality $\Psi(u)=0$ holds if and only if
\[
  \rho_1(u)=\rho_2(u)=u_1.
\]
For the two norms above this is exactly the condition in
\eqref{eq:nonradial-face-example-U}.  The dimension can be seen explicitly by writing
\[
  s:=\frac{u_2}{u_1},
  \qquad
  t:=\frac{u_3}{u_1}.
\]
Then
\[
  |s|\leq\frac12,
  \qquad
  |t|\leq\frac13,
  \qquad
  u=\frac{(1,s,t)}{\sqrt{1+s^2+t^2}},
\]
which parametrizes $\cU$ by the rectangle
$[-1/2,1/2]\times[-1/3,1/3]$.  Hence $\cU$ is two-dimensional.

It remains to verify the uniform expansion.  For $u\in S^2$ and $r>0$,
\begin{align*}
  g(ru)-r\Psi(u)
  &={}
  \sqrt{1+r^2\rho_1(u)^2}-r\rho_1(u)
  +
  \sqrt{1+r^2\rho_2(u)^2}-r\rho_2(u).
\end{align*}
Since $\rho_1(u)\geq 1/\sqrt3$ on $S^2$ and $\rho_2(u)\geq\rho_1(u)$, rationalization gives, uniformly in $u\in S^2$,
\[
  r\bigl(\sqrt{1+r^2\rho_i(u)^2}-r\rho_i(u)\bigr)
  =
  \frac{1}{\sqrt{\rho_i(u)^2+r^{-2}}+\rho_i(u)}
  \longrightarrow
  \frac{1}{2\rho_i(u)}.
\]
Summing the two limits proves \eqref{eq:uniform-facewise-expansion} with
$\nu=1$ and \eqref{eq:nonradial-face-example-kappa}.

The function in \eqref{eq:nonradial-face-example-f} is also nonradial: it cannot be written as $\phi(\rho(x))$ for a norm $\rho$ and a one-variable profile $\phi$.  Indeed, $t\mapsto f(te_1)$ is strictly increasing on $[0,\infty)$, so any such representation would force $\phi$ to be strictly increasing on $[0,\infty)$.  The restrictions to any two rays would therefore differ only by a constant rescaling of the radial variable.  In particular, for some $c>0$ one would have
\[
  f(te_2)=f(cte_1)
  \qquad(t\geq0).
\]
But
\[
  f(te_1)=2\sqrt{1+t^2}
  =2+t^2+O(t^4),
  \qquad
  f(te_2)=\sqrt{1+t^2}+\sqrt{1+4t^2}
  =2+\frac52t^2+O(t^4),
\]
whereas the large-$t$ asymptotics are
\[
  f(te_1)=2t+o(t),
  \qquad
  f(te_2)=3t+o(t).
\]
The large-$t$ comparison would require $c=3/2$, while the quadratic expansion at the origin would require $c^2=5/2$, a contradiction.

The example also makes the directional selection mechanism of
\Cref{thm:facewise-selection} explicit.  Take $h(x):=x_1+x_2+x_3$.
For $u\in\cU$,
\[
  h(u)
  \geq u_1\left(1-\frac12-\frac13\right)
  =\frac{u_1}{6}>0,
\]
so the inward-tilt condition holds.  On $\cU$ we have
$\rho_1(u)=\rho_2(u)=u_1$, hence $\kappa(u)=1/u_1$.  Writing
$s=u_2/u_1$ and $t=u_3/u_1$, the selection functional satisfies
\[
  M(u)^2=\kappa(u)h(u)=1+s+t.
\]
Therefore $M$ has the unique minimizer
\[
  u_*
  =\frac{(1,-1/2,-1/3)}{\sqrt{1+1/4+1/9}}
  =\left(\frac67,-\frac37,-\frac27\right).
\]
Thus a two-dimensional critical face is reduced by the inward tilt to a single escaping direction.  In this case
\[
  M_*=\frac1{\sqrt6},
  \qquad
  m(\mu)\sim\sqrt{\frac23}\,\mu^{1/2},
  \qquad
  r_\mu\sim\frac{7}{\sqrt6}\,\mu^{-1/2}.
\]
\end{example}

Let $h\in X^*$ be an inward tilt satisfying
\begin{equation}\label{eq:inward-positive}
  h(u)>0
  \qquad\text{for every }u\in\cU.
\end{equation}
For $\mu>0$, set $p_\mu:=p_0-\mu h$ and consider the perturbed convex problem
\begin{equation}\label{prob:facewise-tilt}
  m(\mu):=
  \inf_{x\in X}
  \{f(x)-p_\mu(x)+f^*(p_0)\}
  =\inf_{x\in X}\{g(x)+\mu h(x)\}.
\end{equation}
Define on the critical face
\[
  M(u):=\kappa(u)^{1/(\nu+1)}h(u)^{\nu/(\nu+1)},
  \qquad
  M_*:=\min_{u\in\cU}M(u),
\]
and
\[
  \cU_*:=\operatorname*{argmin}_{u\in\cU}M(u).
\]
The exponents in $M$ are the weighted geometric-mean exponents arising from the one-dimensional balance
\[
  \min_{s>0}\{s h(u)+\kappa(u)s^{-\nu}\}=C_\nu M(u),
\]
which is the radial optimization performed in the proof below.  Since $\cU$ is compact and both factors in $M$ are positive and continuous, $M_*>0$ and $\cU_*$ is nonempty and compact.

The inward perturbation selects a direction through both the lower-order defect and the direction of the cost perturbation.

\begin{theorem}\label{thm:facewise-selection}
Let \Cref{ass:facewise-expansion} hold and suppose \eqref{eq:inward-positive} is satisfied.  Then there exists $\mu_0>0$ such that, for every $0<\mu<\mu_0$, the function $g+\mu h$ is coercive and problem \eqref{prob:facewise-tilt} has a nonempty compact minimizer set.

Let $x_\mu$ be any minimizer and, for sufficiently small $\mu$, set $r_\mu:=\norm{x_\mu}_0$ and $u_\mu:=x_\mu/r_\mu$.  Define $\alpha:=\nu/(\nu+1)$ and $C_\nu:=(\nu+1)\nu^{-\nu/(\nu+1)}$.  Then, as $\mu\downarrow0$,
\begin{equation}\label{eq:facewise-value-law}
  m(\mu)
  \sim C_\nu M_*\,\mu^\alpha.
\end{equation}
Moreover,
\begin{equation}\label{eq:facewise-direction-selection}
  \dist(u_\mu,\cU_*)\longrightarrow0,
  \qquad
  \frac{\Psi(u_\mu)}{\mu}\longrightarrow0.
\end{equation}
Every sequence $\mu_j\downarrow0$ has a subsequence for which $u_{\mu_j}\to\bar u\in\cU_*$, and along any such subsequence
\begin{equation}\label{eq:facewise-radius-law}
  \mu_j^{1/(\nu+1)}r_{\mu_j}
  \longrightarrow
  \left(\frac{\nu\kappa(\bar u)}{h(\bar u)}\right)^{1/(\nu+1)}.
\end{equation}
In particular, if $\cU_*=\{u_*\}$ is a singleton, then the whole family satisfies
\begin{equation}\label{eq:facewise-unique-selection}
  u_\mu\longrightarrow u_*,
  \qquad
  r_\mu\sim
  \left(\frac{\nu\kappa(u_*)}{h(u_*)}\right)^{1/(\nu+1)}
  \mu^{-1/(\nu+1)}.
\end{equation}
\end{theorem}

\begin{proof}
We divide the proof into six steps.

\emph{Step 1: positivity of the tilted recession density.}
Since $\cU$ is compact and $h$ is positive on $\cU$, there are a neighborhood $V$ of $\cU$ in $S_0$ and $h_{\min}>0$ such that $h(u)\geq h_{\min}$ for $u\in V$.  On the compact set $S_0\setminus V$, the continuous function $\Psi$ is strictly positive; set $\delta_\Psi:=\min_{S_0\setminus V}\Psi>0$.
Let $H_0:=\max_{S_0}|h|$.  After decreasing $\mu_0$ if necessary, for $0<\mu<\mu_0$ we have $\mu H_0\leq\delta_\Psi/2$.  Therefore
\begin{equation}\label{eq:tilted-density-lower}
  \Theta_\mu(u):=\Psi(u)+\mu h(u)
  \geq
  \begin{cases}
    \mu h_{\min},&u\in V,\\
    \delta_\Psi/2,&u\in S_0\setminus V.
  \end{cases}
\end{equation}
In particular, $\Theta_\mu(u)>0$ on $S_0$.  By the uniform expansion \eqref{eq:uniform-facewise-expansion}, for every sufficiently large $r$,
\begin{equation}\label{eq:coercive-expansion}
  g(ru)+\mu h(ru)
  =r \Theta_\mu(u)+r^{-\nu}\bigl(\kappa(u)+o(1)\bigr),
\end{equation}
where the $o(1)$ is uniform in $u$.  For each fixed $\mu>0$, \eqref{eq:tilted-density-lower} implies that the right-hand side tends to $+\infty$ uniformly in $u$.  Thus $g+\mu h$ is coercive.  Since it is finite and continuous, its minimizer set is nonempty and compact.

\emph{Step 2: an upper bound at the natural scale.}
Fix $u_*\in\cU_*$.  For $s>0$ choose $r:=s\mu^{-1/(\nu+1)}$.
Because $\Psi(u_*)=0$, the uniform expansion gives
\begin{align}
  \mu^{-\alpha}\bigl(g(ru_*)+\mu h(ru_*)\bigr)
  &=s h(u_*)+\kappa(u_*)s^{-\nu}+o(1).
  \label{eq:scaled-candidate}
\end{align}
The right-hand side is minimized at
\[
  s_*(u):=
  \left(\frac{\nu\kappa(u)}{h(u)}\right)^{1/(\nu+1)},
\]
with minimum
\[
  C_\nu\kappa(u)^{1/(\nu+1)}h(u)^{\nu/(\nu+1)}.
\]
Using $u=u_*$ and $s=s_*(u_*)$ in \eqref{eq:scaled-candidate} yields
\begin{equation}\label{eq:m-upper-facewise}
  \limsup_{\mu\downarrow0}
  \mu^{-\alpha}m(\mu)
  \leq C_\nu M_*.
\end{equation}
In particular, $m(\mu)=O(\mu^\alpha)$.

\emph{Step 3: minimizers escape and their radii have the correct order.}
If a sequence of minimizers $x_{\mu_j}$ remained bounded for some $\mu_j\downarrow0$, then, after taking a subsequence, $x_{\mu_j}\to\bar x$.  Optimality and the upper bound just proved give
\[
  0\leq g(x_{\mu_j})
  \leq m(\mu_j)+\mu_j|h(x_{\mu_j})|
  \longrightarrow0.
\]
Hence $g(\bar x)=0$, contradicting $p_0\notin\ran\partial f$.  Thus
\begin{equation}\label{eq:rmu-infty}
  r_\mu\longrightarrow\infty.
\end{equation}
Let $\kappa_-:=\min_{S_0}\kappa>0$.
Uniformity in \eqref{eq:uniform-facewise-expansion} and \eqref{eq:rmu-infty} imply, for all sufficiently small $\mu$,
\begin{equation}\label{eq:basic-lower-r}
  m(\mu)
  \geq r_\mu \Theta_\mu(u_\mu)
       +\frac{\kappa_-}{2}r_\mu^{-\nu}.
\end{equation}
Although $u_\mu$ has not yet been localized to $V$, the global bound
\begin{equation}\label{eq:tilted-density-global}
  \Theta_\mu(u)\geq c_0\mu
  \qquad(u\in S_0)
\end{equation}
holds for small $\mu$, with $c_0:=\min\{h_{\min},\delta_\Psi/(2\mu_0)\}>0$.  Combining \eqref{eq:m-upper-facewise}, \eqref{eq:basic-lower-r}, and \eqref{eq:tilted-density-global} gives constants $0<c<C<\infty$ such that
\begin{equation}\label{eq:rmu-two-sided}
  c\mu^{-1/(\nu+1)}
  \leq r_\mu
  \leq C\mu^{-1/(\nu+1)}
\end{equation}
for all sufficiently small $\mu$.  Indeed, the term $r_\mu^{-\nu}$ gives the lower bound on $r_\mu$, while the term $\mu r_\mu$ gives the upper bound.

\emph{Step 4: localization to the critical face.}
From \eqref{eq:coercive-expansion}, \eqref{eq:rmu-two-sided}, and $m(\mu)=O(\mu^\alpha)$ we obtain
\[
  r_\mu\Psi(u_\mu)
  \leq O(\mu^\alpha)+\mu H_0r_\mu+O(r_\mu^{-\nu})
  =O(\mu^\alpha).
\]
Since $r_\mu\geq c\mu^{-1/(\nu+1)}$, it follows that
\begin{equation}\label{eq:Psi-O-mu}
  \Psi(u_\mu)=O(\mu).
\end{equation}
Hence every cluster point of $(u_\mu)$ lies in $\cU$.

\emph{Step 5: scaled liminf and direction selection.}
Take an arbitrary sequence $\mu_j\downarrow0$.  By compactness of $S_0$ and the two-sided estimate \eqref{eq:rmu-two-sided}, after passing to a subsequence we may assume
\[
  u_{\mu_j}\to\bar u\in\cU,
  \qquad
  s_j:=\mu_j^{1/(\nu+1)}r_{\mu_j}\to\bar s\in(0,\infty),
\]
while, by \eqref{eq:Psi-O-mu},
\[
  \frac{\Psi(u_{\mu_j})}{\mu_j}\to\theta
\]
along a further subsequence for some $\theta\in[0,\infty)$.
Using the uniform expansion and dividing by $\mu_j^\alpha$ gives
\begin{align*}
  \mu_j^{-\alpha}m(\mu_j)
  &=s_j\left(
      \frac{\Psi(u_{\mu_j})}{\mu_j}
      +h(u_{\mu_j})
    \right)
    +\kappa(u_{\mu_j})s_j^{-\nu}+o(1)\\
  &\longrightarrow
    \bar s\bigl(\theta+h(\bar u)\bigr)
    +\kappa(\bar u)\bar s^{-\nu}.
\end{align*}
For every $s>0$, $u\in\cU$, and $\theta\geq0$,
\begin{align}
  s(\theta+h(u))+\kappa(u)s^{-\nu}
  &\geq s h(u)+\kappa(u)s^{-\nu}\\
  &\geq C_\nu M(u)\\
  &\geq C_\nu M_*.
  \label{eq:facewise-lower-chain}
\end{align}
Together with the upper bound \eqref{eq:m-upper-facewise}, this proves
\[
  \mu^{-\alpha}m(\mu)\longrightarrow C_\nu M_*,
\]
which is \eqref{eq:facewise-value-law}.  Equality throughout \eqref{eq:facewise-lower-chain} is now forced for every subsequential limit.  Therefore
\begin{equation}\label{eq:equality-forces}
  \theta=0,
  \qquad
  \bar u\in\cU_*,
  \qquad
  \bar s=s_*(\bar u).
\end{equation}
Since the initial sequence $\mu_j\downarrow0$ was arbitrary, \eqref{eq:equality-forces} yields both statements in \eqref{eq:facewise-direction-selection}, and every cluster direction belongs to $\cU_*$.  It also gives \eqref{eq:facewise-radius-law}.

\emph{Step 6: the uniquely selected direction.}
If $\cU_*=\{u_*\}$, compactness of the sphere and \eqref{eq:facewise-direction-selection} imply $u_\mu\to u_*$.  Then every subsequence of the scaled radii has the same limit $s_*(u_*)$ by \eqref{eq:facewise-radius-law}; hence the entire family converges to that limit.  This proves \eqref{eq:facewise-unique-selection} and completes the proof.
\end{proof}

\begin{remark}\label{rem:facewise-geometry}
The function $\Psi=f^\infty-p_0$ depends only on recession geometry, whereas $\kappa$ records the first lower-order correction.  The selected direction need not minimize the defect or the inward cost separately; it minimizes their weighted geometric mean through $M$.  In the radial case there is only one effective critical direction, so this directional competition is absent.
\end{remark}

\begin{corollary}\label{cor:ellipsoid-facewise}
Let $\delta>0$, let $f_H(z):=\sqrt{\delta^2+z^{\trans}Hz}$ with $H\succ0$, and let $p_0$ satisfy $p_0^{\trans}H^{-1}p_0=1$.  Then $p_0\in\dom f_H^*\setminus\ran\partial f_H$.  For any norm $\norm{\cdot}_0$ on $\R^n$, the expansion \eqref{eq:uniform-facewise-expansion} holds with $\nu=1$,
\begin{equation}\label{eq:ellipsoid-Psi-kappa}
  \Psi(u)=\sqrt{u^{\trans}Hu}-p_0^{\trans}u,
  \qquad
  \kappa(u)=\frac{\delta^2}{2\sqrt{u^{\trans}Hu}}.
\end{equation}
Define the normalized critical direction $u_0:=H^{-1}p_0/\norm{H^{-1}p_0}_0$.  Then the critical set is the singleton $\cU=\{u_0\}$.  Consequently, every inward tilt $p_\mu=p_0-\mu h$ with $h(u_0)>0$ produces optimizers satisfying $u_\mu\to u_0$, while $\norm{x_\mu}_0$ grows like $\mu^{-1/2}$ with the constant given by \eqref{eq:facewise-unique-selection}.
\end{corollary}

\begin{proof}
To verify the boundary membership directly, write $\rho_{H,*}(p)=\sqrt{p^{\trans}H^{-1}p}$.  Cauchy--Schwarz in the $H$ metric reduces the conjugate to a one-dimensional maximization and gives
\[
  f_H^*(p)=
  \begin{cases}
  -\delta\sqrt{1-\rho_{H,*}(p)^2},&\rho_{H,*}(p)\leq1,\\
  +\infty,&\rho_{H,*}(p)>1.
  \end{cases}
\]
For $\rho_{H,*}(p)<1$ the supremum is attained at a finite point, whereas for $\rho_{H,*}(p)=1$ its value is $0$ and is approached only along the positive ray generated by $H^{-1}p$.  Thus $p_0^{\trans}H^{-1}p_0=1$ implies $p_0\in\dom f_H^*\setminus\ran\partial f_H$.  Uniformly for $u\in S_0$,
\begin{align*}
  \sqrt{\delta^2+r^2u^{\trans}Hu}
  &=r\sqrt{u^{\trans}Hu}
    +\frac{\delta^2}{2r\sqrt{u^{\trans}Hu}}
    +O(r^{-3}),
\end{align*}
where uniformity follows because $H\succ0$ and $S_0$ is compact.  Subtracting $rp_0^{\trans}u$ gives \eqref{eq:ellipsoid-Psi-kappa}.  Equality in
\[
  p_0^{\trans}u
  \leq\sqrt{p_0^{\trans}H^{-1}p_0}\sqrt{u^{\trans}Hu}
  =\sqrt{u^{\trans}Hu}
\]
holds precisely when $u$ is a positive multiple of $H^{-1}p_0$.  Intersecting this ray with $S_0$ gives the unique critical direction $u_0$.  The conclusion follows from \Cref{thm:facewise-selection} with $\nu=1$.
\end{proof}

\section{Radial attainment boundary}\label{sec:radial}

Let $\rho$ be a norm on $X$ with dual norm $\rho_*$, and assume the scalar profile satisfies the following.

\begin{assumption}\label{ass:profile}
The function $\phi:[0,\infty)\to\R$ is finite, convex, and nondecreasing, and $L:=\lim_{r\to\infty}\phi(r)/r\in(0,\infty)$.
\end{assumption}

Define $C_{\phi,\rho}:=\{(t,z):t\geq\phi(\rho(z))\}$ and $\psi(r):=\phi(r)-Lr$.
Convexity implies
\[
 0\leq\phi(r_2)-\phi(r_1)\leq L(r_2-r_1),
 \qquad 0\leq r_1<r_2,
\]
so $\psi$ is nonincreasing; write $\gamma:=\lim_{r\to\infty}\psi(r)\in[-\infty,\infty)$.
The radial structure makes the entire boundary classification one-dimensional.

\begin{theorem}\label{thm:boundary-trichotomy}
Under \Cref{ass:profile}, $C_{\phi,\rho}^\infty=\{(s,w):s\geq L\rho(w)\}$.  Let $\ell_{a,b}(t,z)=at+b(z)$, assume $a>0$, and put $\beta:=\rho_*(b)$.  Then
\begin{equation}\label{eq:scalar-reduction}
 \inf_{C_{\phi,\rho}}\ell_{a,b}
 =\inf_{r\geq0}\{a\phi(r)-\beta r\}.
\end{equation}
Consequently:
\begin{enumerate}[label=\textup{(\roman*)}]
\item if $\beta>aL$, the value is $-\infty$;
\item if $\beta<aL$, the scalar objective is coercive and the minimum is attained;
\item if $\beta=aL$, the value is $a\gamma$.  If $\gamma=-\infty$, the problem is unbounded.  If $\gamma\in\R$, the value is attained if and only if $\psi$ reaches $\gamma$ at a finite radius, equivalently if and only if $\phi$ is eventually affine with slope $L$.  Otherwise the value is finite and unattained.
\end{enumerate}
Thus copositivity on the recession cone is equivalent to $aL\geq\beta$, but at equality it does not decide boundedness or attainment.
\end{theorem}

\begin{proof}
The recession formula follows from the recession function of $z\mapsto\phi(\rho(z))$, which is $L\rho(z)$.  For fixed $r=\rho(z)$,
\[
 \min_{\rho(z)=r}b(z)=-r\rho_*(b),
\]
which proves \eqref{eq:scalar-reduction}.  Since
\[
 \frac{a\phi(r)-\beta r}{r}\to aL-\beta,
\]
the strict cases follow.  At $\beta=aL$, the scalar objective is $a\psi(r)$.  Since $\psi$ is nonincreasing, its infimum is $\gamma$; if a finite $r_0$ satisfies $\psi(r_0)=\gamma$, monotonicity forces $\psi\equiv\gamma$ on $[r_0,\infty)$, which is exactly eventual affinity.
\end{proof}

The hyperbolic profile $\phi(r)=\sqrt{\delta^2+r^2}$, $L=1$, has
\begin{equation}\label{eq:hyperbolic-defect}
 \phi(r)-r
 =\frac{\delta^2}{\sqrt{\delta^2+r^2}+r}.
\end{equation}
This quantity tends to $0^+$ as $r\to\infty$ but is strictly positive for every finite $r$, so the limiting defect is not attained.  Hence every nonzero boundary objective has finite unattained infimum.  By contrast, $\phi(r)=\max\{\delta,r\}$ reaches its limiting defect and attains boundary optima, while $\phi(r)=r-\log(1+r)$ has defect tending to $-\infty$.  These profiles have the same recession cone, illustrating exactly what recession geometry omits.

\section{Universal scaling laws at the attainment boundary}\label{sec:universal-rates}

Assume from now on that \Cref{ass:profile} holds and that the defect has a finite unattained limit $\gamma$.  To quantify the divergence of optimizing decisions near the boundary, set
\begin{equation}\label{eq:positive-defect}
  d(r):=\phi(r)-Lr-\gamma>0,
  \qquad d(r)\downarrow0.
\end{equation}
The function $d$ is convex because it differs from $\phi$ by an affine function.

The scaling laws are governed by regular variation.  We assume that for some $\nu>0$ and $\kappa>0$,
\begin{equation}\label{eq:regular-defect}
  d(r)\sim\kappa r^{-\nu}
  \qquad(r\to\infty).
\end{equation}
No differentiability assumption is needed.  Convexity transfers the asymptotic law to the one-sided derivatives, as in the Monotone Density Theorem for regularly varying functions; see \cite[Theorem~1.7.2]{BGT1987}.  The secant argument below applies directly to a nonsmooth convex $d$ and gives both one-sided derivative asymptotics.

Recent work also gives error bounds for nonpolyhedral cones.  Facial-residual methods yield sharp, and sometimes non-H\"older, distance estimates for exponential, $p$-, and log-determinant cones \cite{LLP2025,LLLP2025}; for perspective cones, Wang, Louren\c{c}o, and Pong \cite{WLP2025} identify broad circumstances in which a one-half H\"older exponent is generic.  The exponents here come from the tail $d(r)\sim\kappa r^{-\nu}$ and describe the size of optimizing decisions and value residuals under objective tilting or regularization.  For the hyperbolic profile $\nu=1$, tilting gives the same square-root scale, but for an escaping optimal path rather than distance to feasibility.

\begin{lemma}\label{lem:convex-regular-derivative}
Suppose $d$ is convex, positive, nonincreasing, and satisfies \eqref{eq:regular-defect}.  Then
\begin{equation}\label{eq:derivative-regular}
  d'_-(r)\sim -\nu\kappa r^{-\nu-1},
  \qquad
  d'_+(r)\sim -\nu\kappa r^{-\nu-1}.
\end{equation}
\end{lemma}

\begin{proof}
Fix $\vartheta>1$.  Convexity gives
\[
  \frac{d(r)-d(r/\vartheta)}{r-r/\vartheta}
  \leq d'_-(r)\leq d'_+(r)
  \leq\frac{d(\vartheta r)-d(r)}{(\vartheta-1)r}.
\]
After multiplication by $r^{\nu+1}/\kappa$ and passage to infinity, the left and right secants converge respectively to
\[
  -\frac{\vartheta(\vartheta^\nu-1)}{\vartheta-1}
  \qquad\text{and}\qquad
  -\frac{1-\vartheta^{-\nu}}{\vartheta-1}.
\]
Letting $\vartheta\downarrow1$ squeezes both one-sided derivatives to $-\nu$ after normalization.
\end{proof}

For objective tilting, fix $a>0$ and, for $\mu>0$, consider a radial objective whose scalar reduction is
\[
  G_\mu(r):=a\phi(r)-(aL-\mu)r
  =a\gamma+a d(r)+\mu r.
\]
This corresponds to decreasing the dual-norm magnitude of the horizontal cost by the margin $\mu$.

\begin{theorem}\label{thm:universal-tilt}
Assume \eqref{eq:regular-defect}.  Let $r_\mu$ be any minimizer of $G_\mu$.  Then, as $\mu\downarrow0$,
\begin{align}
  r_\mu
  &\sim\left(\frac{a\nu\kappa}{\mu}\right)^{1/(\nu+1)},
  \label{eq:universal-tilt-radius}\\
  \min_{r\geq0}G_\mu(r)-a\gamma
  &\sim C_\nu(a\kappa)^{1/(\nu+1)}
  \mu^{\nu/(\nu+1)},
  \label{eq:universal-tilt-value}\\
  a d(r_\mu)
  &\sim \nu^{-\nu/(\nu+1)}(a\kappa)^{1/(\nu+1)}
  \mu^{\nu/(\nu+1)},
  \label{eq:universal-tilt-residual}
\end{align}
where $C_\nu:=(\nu+1)\nu^{-\nu/(\nu+1)}$.  Every family of minimizers has the same asymptotic radius.
\end{theorem}

\begin{proof}
The function $G_\mu$ is convex and coercive, so minimizers exist.  Any bounded family of minimizers would, after passage to the limit $\mu\downarrow0$, minimize $d$ at a finite point, contrary to \eqref{eq:positive-defect}.  Hence $r_\mu\to\infty$.  The optimality condition is
\[
  -\frac{\mu}{a}\in\partial d(r_\mu)
  =[d'_-(r_\mu),d'_+(r_\mu)].
\]
By \Cref{lem:convex-regular-derivative},
\[
  \mu\sim a\nu\kappa r_\mu^{-\nu-1},
\]
which proves \eqref{eq:universal-tilt-radius}.  Since $d(r_\mu)\sim\kappa r_\mu^{-\nu}$ and
\[
  \mu r_\mu\sim a\nu\kappa r_\mu^{-\nu},
\]
the excess value is asymptotic to $a(\nu+1)\kappa r_\mu^{-\nu}$.  Substitution gives \eqref{eq:universal-tilt-value}, and the same calculation without the linear tilt gives \eqref{eq:universal-tilt-residual}.
\end{proof}

A hard decision budget gives a second approximation mechanism.  At a nonzero boundary objective, impose $\rho(z)\leq R$ and denote the restricted value by $v_R$.

\begin{theorem}\label{thm:universal-truncation}
Let $a>0$ and $\rho_*(b)=aL$.  Then for every $R\geq0$,
\begin{equation}\label{eq:truncated-exact}
  v_R=a\gamma+a d(R).
\end{equation}
If \eqref{eq:regular-defect} holds, then
\[
  v_R-a\gamma\sim a\kappa R^{-\nu}.
\]
Consequently, defining $R_\eta:=\inf\{R\geq0:v_R-a\gamma\leq\eta\}$, the smallest admissible budget satisfies
\[
  R_\eta\sim\left(\frac{a\kappa}{\eta}\right)^{1/\nu}
  \qquad(\eta\downarrow0).
\]
\end{theorem}

\begin{proof}
The scalar boundary objective is $a\gamma+a d(r)$, which is nonincreasing.  Its minimum over $0\leq r\leq R$ is therefore attained at $R$, proving \eqref{eq:truncated-exact}.  The remaining statements follow by regular variation and inversion of the monotone function $d$.
\end{proof}

For power regularization, let $q>1$ and $\lambda>0$ and consider
\[
  \inf_{(t,z)\in C_{\phi,\rho}}
  \left\{at+b(z)+\frac{\lambda}{q}\rho(z)^q\right\},
  \qquad \rho_*(b)=aL.
\]
Its scalar reduction is
\begin{equation}\label{eq:power-scalar}
  J_\lambda(r)
  =a\gamma+a d(r)+\frac{\lambda}{q}r^q.
\end{equation}

\begin{theorem}\label{thm:universal-regularization}
Assume \eqref{eq:regular-defect}, and let $r_\lambda$ be any minimizing radius in \eqref{eq:power-scalar}.  Then
\begin{align}
  r_\lambda
  &\sim\left(\frac{a\nu\kappa}{\lambda}\right)^{1/(\nu+q)},
  \label{eq:power-radius}\\
  a d(r_\lambda)
  &\sim a\kappa
  \left(\frac{\lambda}{a\nu\kappa}\right)^{\nu/(\nu+q)},
  \label{eq:power-residual}\\
  \min J_\lambda-a\gamma
  &\sim\left(1+\frac{\nu}{q}\right)a\kappa
  \left(\frac{\lambda}{a\nu\kappa}\right)^{\nu/(\nu+q)}.
  \label{eq:power-value}
\end{align}
In particular, the regularization path escapes at rate $\lambda^{-1/(\nu+q)}$.
\end{theorem}

\begin{proof}
Existence follows from coercivity.  As before, minimizing radii diverge because the unregularized boundary infimum is unattained.  The scalar optimality condition is
\[
  -\frac{\lambda}{a}r_\lambda^{q-1}
  \in\partial d(r_\lambda).
\]
Using \eqref{eq:derivative-regular} yields
\[
  \lambda r_\lambda^{q-1}
  \sim a\nu\kappa r_\lambda^{-\nu-1},
\]
which proves \eqref{eq:power-radius}.  Regular variation gives \eqref{eq:power-residual}.  The stationarity asymptotic also gives
\[
  \frac{\lambda}{q}r_\lambda^q
  \sim\frac{a\nu\kappa}{q}r_\lambda^{-\nu}.
\]
Adding the residual and penalty terms proves \eqref{eq:power-value}.
\end{proof}

The three mechanisms produce different exponents even though they approximate the same boundary value.

\begin{table}[t]
\centering
\caption{Universal scales for a defect $d(r)\sim\kappa r^{-\nu}$.}
\label{tab:universal-scales}
\begin{tabular}{@{}lll@{}}
\toprule
Approximation mechanism & selected radius & boundary residual or value excess\\
\midrule
Objective margin $\mu$ & $\mu^{-1/(\nu+1)}$ & $\mu^{\nu/(\nu+1)}$\\
Hard error tolerance $\eta$ & $\eta^{-1/\nu}$ & $\eta$\\
Power penalty $\lambda r^q/q$ & $\lambda^{-1/(\nu+q)}$ & $\lambda^{\nu/(\nu+q)}$\\
\bottomrule
\end{tabular}
\end{table}

\begin{remark}\label{rem:hyperbolic-universality}
For the hyperbolic profile,
\[
  \sqrt{\delta^2+r^2}-r
  \sim\frac{\delta^2}{2r},
\]
so $\nu=1$ and $\kappa=\delta^2/2$.  This is the same exponent $\nu=1$ that appears in the facewise expansion of \Cref{cor:ellipsoid-facewise}; thus the nonradial selection theorem and the scalar radial laws use the same first correction for the ellipsoidal example.  Objective tilting therefore gives the inverse square-root radius law, a hard residual tolerance gives the inverse-radius law, and quadratic regularization gives the cube-root law.  The exact SOCP formulas below refine these general asymptotics and include the influence of the objective and ellipsoidal metric.
\end{remark}

\section{Ellipsoidal second-order cone programs}\label{sec:socp-models}

For shifted Lorentz-cone geometry, the radial theory gives closed primal and dual formulas, including the boundary case in which strict feasibility and dual attainment coexist with primal nonattainment.

Let $H\in\R^{n\times n}$ be symmetric positive definite, let $H^{1/2}$ denote its symmetric positive-definite square root, and let $\delta>0$.  Define
\[
  C_{\delta,H}:=
  \left\{(t,z)\in\R\times\R^n:
  t\geq\sqrt{\delta^2+z^{\trans}Hz}\right\}.
\]
This is the hyperbolic norm epigraph associated with $\rho_H(z):=\sqrt{z^{\trans}Hz}$, whose dual norm is $\rho_{H,*}(b)=\sqrt{b^{\trans}H^{-1}b}$.
Define
\begin{equation}\label{eq:betaH}
  \beta_H(b):=\sqrt{b^{\trans}H^{-1}b},
  \qquad \beta:=\beta_H(b).
\end{equation}
The optimization problem is
\begin{equation}\label{prob:ellipsoid}
  (P_{a,b}^{\delta,H})
  \qquad
  \inf_{(t,z)\in C_{\delta,H}}\bigl(at+b^{\trans}z\bigr).
\end{equation}

The general radial theory now yields a complete phase diagram for the ellipsoidal SOCP.

\begin{theorem}\label{thm:ellipsoidal-socp}
The recession cone is
\begin{equation}\label{eq:recH}
  C_{\delta,H}^\infty
  =\{(s,w):s\geq\sqrt{w^{\trans}Hw}\}.
\end{equation}
The objective is copositive on \eqref{eq:recH} if and only if $a\geq\beta$.

\begin{enumerate}[label=\textup{(\roman*)}]
\item If $a<\beta$, then $\val(P_{a,b}^{\delta,H})=-\infty$.
\item If $a>\beta$ and $\Delta=\sqrt{a^2-\beta^2}$, then the unique minimizer is
\begin{equation}\label{eq:ellipsoid-minimizer}
  t^*=\frac{\delta a}{\Delta},
  \qquad
  z^*=-\frac{\delta}{\Delta}H^{-1}b,
\end{equation}
and $\val(P_{a,b}^{\delta,H})=\delta\Delta$.
\item If $a=\beta>0$, then the value is zero and is not attained.
\item If $(a,b)=(0,0)$, every feasible point is optimal.
\end{enumerate}
\end{theorem}

\begin{proof}
If $a<0$, the objective is unbounded below by sending $t\to+\infty$.  If $a=0$ and $b\neq0$, it is unbounded below by moving $z$ along a direction on which $b^{\trans}z\to-\infty$; the case $(a,b)=(0,0)$ is immediate.  We may therefore assume $a>0$.

Take $\rho=\rho_H$ and $\phi(r)=\sqrt{\delta^2+r^2}$ in \Cref{thm:boundary-trichotomy}.  Since $L=1$ and $\rho_{H,*}(b)=\beta$, the recession cone, copositivity criterion, and the cases $0<a<\beta$ and $a=\beta$ follow immediately.  Suppose $a>\beta$.  For fixed $r=\rho_H(z)$, the least possible value of $b^{\trans}z$ is $-\beta r$; when $b\neq0$ it is attained uniquely at
\[
  z=-\frac{r}{\beta}H^{-1}b.
\]
Thus it remains to minimize
\[
  \chi(r)=a\sqrt{\delta^2+r^2}-\beta r,\qquad r\geq0.
\]
If $b=0$, then $r=0$ is the unique minimizer.  If $b\neq0$, strict convexity and $\chi'(r)=0$ give
\[
  \frac{ar}{\sqrt{\delta^2+r^2}}=\beta,
  \qquad
  r=\frac{\delta\beta}{\sqrt{a^2-\beta^2}}.
\]
Substituting this radius yields \eqref{eq:ellipsoid-minimizer} and
$\val(P_{a,b}^{\delta,H})=\delta\sqrt{a^2-\beta^2}$.
\end{proof}

The SOCP representation is
\[
  (t,H^{1/2}z,\delta)\in\Qcone^{n+2},
  \qquad
  \Qcone^{n+2}:=
  \{(y_0,\bar y):y_0\geq\norm{\bar y}\}.
\]
In particular, $(2\delta,0)$ is strictly feasible.  The dual can be written explicitly.

Write
\[
  A(t,z):=(t,H^{1/2}z,0),
  \qquad
  \omega_\delta:=(0,0,-\delta),
\]
so that the constraint is $A(t,z)-\omega_\delta\in\Qcone^{n+2}$.  Since the Lorentz cone is self-dual, the Lagrange dual is
\[
\begin{aligned}
  (D_{a,b}^{\delta,H})\qquad
  \sup_{y\in\Qcone^{n+2}}\quad &\ip{\omega_\delta}{y}\\
  \text{subject to}\quad&A^{\trans}y=(a,b).
\end{aligned}
\]
Writing $y=(y_0,\bar y,\upsilon)$, the equality constraint gives $y_0=a$ and $\bar y=H^{-1/2}b$.
Therefore the dual reduces to
\begin{equation}\label{eq:dual-scalar}
  \sup\left\{-\delta\upsilon:
  a\geq\sqrt{\beta^2+\upsilon^2}\right\}.
\end{equation}

At the boundary, the conic dual is attained and the duality gap is zero even though the primal optimum is not attained.

\begin{proposition}\label{prop:ellipsoid-duality}
The primal problem \eqref{prob:ellipsoid} satisfies Slater's condition for every objective.  If $a\geq\beta$, the dual has the unique solution $y^*=\left(a,H^{-1/2}b,-\sqrt{a^2-\beta^2}\right)$ and
\[
  \val(D_{a,b}^{\delta,H})=\delta\sqrt{a^2-\beta^2}=\val(P_{a,b}^{\delta,H}).
\]
At the boundary $a=\beta>0$, the dual optimum is attained at $y^*=(\beta,H^{-1/2}b,0)$, while the primal optimum does not exist.
\end{proposition}

\begin{proof}
Strict feasibility follows from $(2\delta,0,\delta)\in\interior\Qcone^{n+2}$.  Formula \eqref{eq:dual-scalar} gives the unique maximizing value
\[
  \upsilon^*=-\sqrt{a^2-\beta^2}
\]
whenever $a\geq\beta$.  The resulting dual value agrees with \Cref{thm:ellipsoidal-socp}.  In the boundary case, the square root is zero and the primal value is finite but unattained.
\end{proof}

Slater's condition does not prevent this form of primal nonattainment.

\begin{remark}\label{rem:slater}
The boundary example satisfies the standard primal constraint qualification.  Slater's condition yields equality of the primal and dual values and dual attainment under finiteness, but it does not force primal attainment.  At $a=\beta>0$, \Cref{prop:ellipsoid-image} shows that the objective image is $(0,\infty)$, so the failure of attainment is exactly the failure of closedness at its lower endpoint; compare \cite{P2007}.
\end{remark}

The one-dimensional objective image gives an equivalent description.

\begin{proposition}\label{prop:ellipsoid-image}
Assume $a\geq0$.  Then
\[
  \{at+b^{\trans}z:(t,z)\in C_{\delta,H}\}
  =
  \begin{cases}
  \R,&0\leq a<\beta,\\
  (0,\infty),&a=\beta>0,\\
  [\delta\sqrt{a^2-\beta^2},\infty),&a>\beta,\\
  \{0\},&(a,b)=(0,0).
  \end{cases}
\]
Thus the image is nonclosed precisely at the nonzero copositive boundary.
\end{proposition}

\begin{proof}
Assume first that $a>0$.  Because $C_{\delta,H}$ is convex, its image under the linear objective is an interval; because $t$ may be increased arbitrarily while preserving feasibility, that interval is unbounded above.  \Cref{thm:ellipsoidal-socp} identifies its infimum and determines whether the left endpoint is attained.  This gives the first three cases.  If $a=0$ and $b\neq0$, the map $z\mapsto b^{\trans}z$ is onto $\R$, while $t$ can always be chosen feasible.  The zero objective is immediate.
\end{proof}

\section{Quantitative behavior near the nonattained boundary}\label{sec:quantitative}

The phase diagram locates the nonattainment boundary; the explicit formulas also describe the approach to it.  Constraint-qualification-free optimality conditions for linear second-order cone programs provide a related perspective on boundary behavior \cite{KT2024}.  For the ellipsoidal model, they give asymptotic rates for the optimizer and the value under inward objective perturbations and quadratic regularization.  Together with the hard-truncation law in \Cref{thm:universal-truncation}, this yields three approximation scales.

\subsection{Sensitivity and ill-conditioning}\label{sec:sensitivity}

Fix $b\neq0$ and let $\beta$ be as in \eqref{eq:betaH}.  We approach the boundary from the strictly copositive side by setting
\[
  a_\tau:=\beta+\tau,
  \qquad \tau>0,
\]
and
\[
  \Delta_\tau:=\sqrt{a_\tau^2-\beta^2}
  =\sqrt{2\beta\tau+\tau^2}.
\]
Let $(t_\tau,z_\tau)$ be the unique minimizer of $P_{a_\tau,b}^{\delta,H}$.  Theorem~\ref{thm:sensitivity} gives the exact blow-up and value laws as the copositivity margin vanishes.

\begin{theorem}\label{thm:sensitivity}
As $\tau\downarrow0$,
\begin{align*}
  \val(P_{a_\tau,b}^{\delta,H})
  &=\delta\Delta_\tau
   =\delta\sqrt{2\beta\tau}\left(1+O(\tau/\beta)\right),\\
  \rho_H(z_\tau)
  &=\frac{\delta\beta}{\Delta_\tau}
   =\delta\sqrt{\frac{\beta}{2\tau}}
     \left(1+O(\tau/\beta)\right),\\
  t_\tau
  &=\frac{\delta(\beta+\tau)}{\Delta_\tau}
   =\delta\sqrt{\frac{\beta}{2\tau}}
     \left(1+O(\tau/\beta)\right).
\end{align*}
Moreover, if $\ell_0(t,z):=\beta t+b^{\trans}z$ is the boundary objective, then
\[
  \ell_0(t_\tau,z_\tau)
  =\frac{\beta\delta\tau}{\Delta_\tau}
  =\delta\sqrt{\frac{\beta\tau}{2}}
   \left(1+O(\tau/\beta)\right).
\]
Thus a small objective tilt of size $\tau$ produces an $O(\sqrt\tau)$ boundary residual but an $O(\tau^{-1/2})$ decision magnitude.  Here $\beta>0$ is fixed; consequently $O(\tau/\beta)$ and $O(\tau)$ are equivalent as $\tau\downarrow0$, but the dimensionless form $O(\tau/\beta)$ records the natural relative margin.
\end{theorem}

\begin{proof}
Insert $a_\tau=\beta+\tau$ into the formulas of \Cref{thm:ellipsoidal-socp} and use
\[
  \Delta_\tau=\sqrt{2\beta\tau+\tau^2}
  =\sqrt{2\beta\tau}\,\left(1+O(\tau/\beta)\right).
\]
The boundary residual follows directly from
\[
  \ell_0(t_\tau,z_\tau)
  =\frac{\beta\delta(\beta+\tau)-\delta b^{\trans}H^{-1}b}{\Delta_\tau}
  =\frac{\beta\delta\tau}{\Delta_\tau}.
\]
\end{proof}

To quantify local conditioning, set $y:=H^{1/2}z$ and $\widehat b:=H^{-1/2}b$, so $\norm{\widehat b}=\beta$.  After eliminating $t$, the reduced objective is $F_a(y):=a\sqrt{\delta^2+\norm{y}^2}+\widehat b^{\trans}y$, whose minimizer for $a>\beta$ is $y_a^*:=-\delta\widehat b/\sqrt{a^2-\beta^2}$.

The same degeneration appears in the spectrum of the reduced Hessian.

\begin{proposition}\label{prop:hessian}
Let $a>\beta>0$ and $\Delta:=\sqrt{a^2-\beta^2}$.  At $y_a^*$, the Hessian of $F_a$ has eigenvalue $\lambda_{\parallel}=\Delta^3/(\delta a^2)$ in the direction of $\widehat b$ and eigenvalue $\lambda_{\perp}=\Delta/\delta$ on $\widehat b^\perp$.  For $n\geq2$, the Euclidean condition number is
\[
  \cond\bigl(\nabla^2F_a(y_a^*)\bigr)
  =\frac{a^2}{\Delta^2}
  =\frac{a^2}{a^2-\beta^2}.
\]
With $a=\beta+\tau$,
\[
  \cond\bigl(\nabla^2F_{a_\tau}(y_{a_\tau}^*)\bigr)
  \sim\frac{\beta}{2\tau}.
\]
\end{proposition}

\begin{proof}
For $\zeta(y)=\sqrt{\delta^2+\norm{y}^2}$,
\[
  \nabla^2\zeta(y)=\frac1{\zeta(y)}\Id-\frac{yy^{\trans}}{\zeta(y)^3}.
\]
At $y_a^*$ one has $\zeta(y_a^*)=\delta a/\Delta$ and $\norm{y_a^*}=\delta\beta/\Delta$.  Thus $a\nabla^2\zeta(y_a^*)$ has eigenvalue $\Delta/\delta$ on $\widehat b^\perp$ and $\Delta^3/(\delta a^2)$ in the direction of $\widehat b$.  Their ratio is $a^2/\Delta^2$, and substituting $a=\beta+\tau$ gives the stated asymptotic condition number.
\end{proof}

\begin{remark}
The smallest curvature decays like $\tau^{3/2}$ along the escaping direction, while transverse curvature decays like $\tau^{1/2}$.  The optimal value alone does not reflect this anisotropy; near the boundary, large optimizer norms are accompanied by severe local ill-conditioning.
\end{remark}

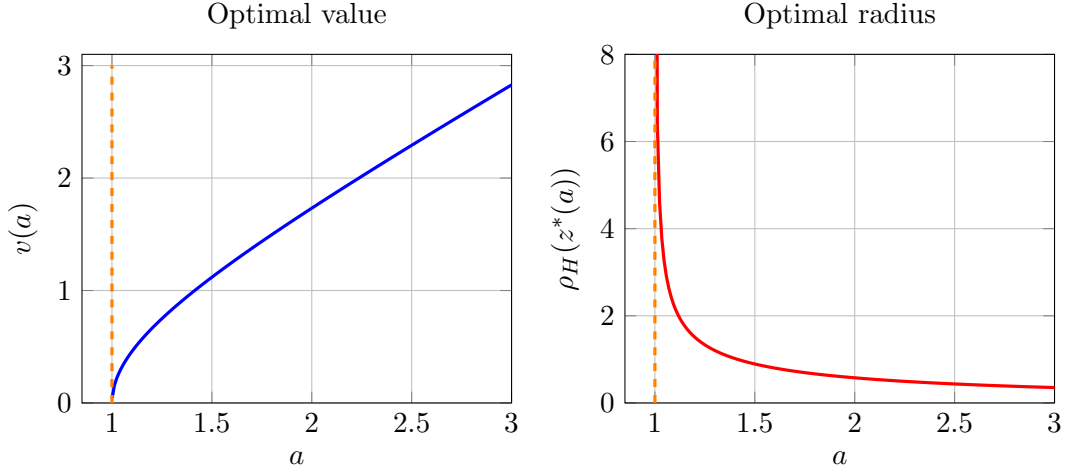
\begin{figure}[t]
\centering
\begin{tikzpicture}
\begin{groupplot}[
  group style={group size=2 by 1,horizontal sep=1.5cm},
  width=.44\textwidth,
  height=6.2cm,
  grid=major,
  xmin=.85,xmax=3,
  samples=180,
  domain=1.001:3,
  xlabel={$a$}
]
\nextgroupplot[ymin=0,ymax=3.1,ylabel={$v(a)$},title={Optimal value}]
\addplot[very thick,blue] {sqrt(x^2-1)};
\addplot[very thick,dashed,orange] coordinates {(1,0) (1,3)};
\nextgroupplot[ymin=0,ymax=8,ylabel={$\rho_H(z^*(a))$},title={Optimal radius}]
\addplot[very thick,red] {1/sqrt(x^2-1)};
\addplot[very thick,dashed,orange] coordinates {(1,0) (1,8)};
\end{groupplot}
\end{tikzpicture}
\caption{For $\delta=\beta=1$, the orange dashed vertical line marks the copositive boundary $a=\beta=1$.  The blue curve shows the optimal value $v(a)=\sqrt{a^2-1}$, which vanishes as the boundary is approached from the right, while the red curve shows the optimal radius $1/\sqrt{a^2-1}$, which diverges.}
\label{fig:phase-transition}
\end{figure}

\subsection{Tikhonov regularization}\label{sec:regularization}

We remain in the ellipsoidal setting and fix the boundary objective $\ell_0(t,z):=\beta t+b^{\trans}z$, where $\beta:=\sqrt{b^{\trans}H^{-1}b}>0$.  For $\eps>0$, consider the geometry-adapted Tikhonov problem
\begin{equation}\label{prob:tikhonov}
  (\mathsf T_\eps)\qquad
  \min_{(t,z)\in C_{\delta,H}}
  \left\{
  \ell_0(t,z)+\frac\eps2\bigl(t^2+z^{\trans}Hz\bigr)
  \right\}.
\end{equation}
The quadratic term is strongly convex in the transformed coordinates $(t,H^{1/2}z)$, so a unique minimizer exists.

Set $\widehat b:=H^{-1/2}b$, $u:=\widehat b/\beta$, and $y:=H^{1/2}z$.  For fixed $s=\norm{y}$, the linear term is minimized by $y=-su$, and for fixed $y$, the optimal $t$ lies on the boundary.  Thus \eqref{prob:tikhonov} reduces to
\[
  \min_{s\geq0}J_\eps(s),
  \qquad
  J_\eps(s):=
  \beta\bigl(\sqrt{\delta^2+s^2}-s\bigr)
  +\frac\eps2(\delta^2+2s^2).
\]

The scalar reduction has a unique critical point and gives the Tikhonov escape asymptotics.

\begin{theorem}\label{thm:tikhonov}
For every $\eps>0$, problem \eqref{prob:tikhonov} has a unique minimizer $(t_\eps,z_\eps)$.  There is a unique $s_\eps>0$ satisfying
\begin{equation}\label{eq:regularized-stationarity}
  \beta\left(
  \frac{s_\eps}{\sqrt{\delta^2+s_\eps^2}}-1
  \right)+2\eps s_\eps=0,
\end{equation}
and $t_\eps=\sqrt{\delta^2+s_\eps^2}$, $z_\eps=-s_\eps H^{-1/2}u$.  As $\eps\downarrow0$,
\begin{align}
  s_\eps
  &\sim\left(\frac{\beta\delta^2}{4\eps}\right)^{1/3},
  \label{eq:s-eps-asymptotic}\\
  \ell_0(t_\eps,z_\eps)
  &\sim 2^{-1/3}(\beta\delta^2)^{2/3}\eps^{1/3},
  \notag\\
  \frac\eps2\bigl(t_\eps^2+z_\eps^{\trans}Hz_\eps\bigr)
  &\sim 2^{-4/3}(\beta\delta^2)^{2/3}\eps^{1/3},
  \notag\\
  \val(\mathsf T_\eps)
  &\sim 3\cdot2^{-4/3}(\beta\delta^2)^{2/3}\eps^{1/3}.
  \label{eq:regularized-value-asymptotic}
\end{align}
In particular, the regularized minimizers escape to infinity at rate $\eps^{-1/3}$.
\end{theorem}

\begin{proof}
The objective in \eqref{prob:tikhonov} is strongly convex and coercive on a nonempty closed convex set, hence has a unique minimizer.  The directional reduction above is exact.  Moreover,
\[
  J_\eps''(s)
  =\frac{\beta\delta^2}{(\delta^2+s^2)^{3/2}}+2\eps>0,
\]
so $J_\eps$ is strictly convex.  Since $J_\eps'(0)=-\beta<0$ and $J_\eps'(s)\to\infty$, it has a unique positive critical point, which is characterized by \eqref{eq:regularized-stationarity}.

First, $s_\eps\to\infty$.  If a sequence $s_{\eps_k}$ remained bounded as $\eps_k\downarrow0$, then \eqref{eq:regularized-stationarity} would imply
\[
  \frac{s}{\sqrt{\delta^2+s^2}}=1
\]
at a finite limit point, which is impossible.  Rewrite the stationarity equation as
\[
  2\eps s_\eps
  =\beta\left(1-\frac{s_\eps}{\sqrt{\delta^2+s_\eps^2}}\right)
  =\frac{\beta\delta^2}
  {\sqrt{\delta^2+s_\eps^2}
   \bigl(\sqrt{\delta^2+s_\eps^2}+s_\eps\bigr)}.
\]
Since $s_\eps\to\infty$, multiplying by $2s_\eps^2/(\beta\delta^2)$ yields
\[
  \frac{4\eps s_\eps^3}{\beta\delta^2}
  =\frac{2s_\eps^2}
  {\sqrt{\delta^2+s_\eps^2}
   \bigl(\sqrt{\delta^2+s_\eps^2}+s_\eps\bigr)}
  \longrightarrow1.
\]
This proves \eqref{eq:s-eps-asymptotic}.

Next,
\[
  \sqrt{\delta^2+s^2}-s
  \sim\frac{\delta^2}{2s},
  \qquad
  \frac12(\delta^2+2s^2)\sim s^2.
\]
Substituting \eqref{eq:s-eps-asymptotic} gives
\begin{align*}
  \beta\bigl(\sqrt{\delta^2+s_\eps^2}-s_\eps\bigr)
  &\sim
  \frac{\beta\delta^2}{2}
  \left(\frac{4\eps}{\beta\delta^2}\right)^{1/3}
  =2^{-1/3}(\beta\delta^2)^{2/3}\eps^{1/3},\\
  \frac\eps2(\delta^2+2s_\eps^2)
  &\sim
  \eps\left(\frac{\beta\delta^2}{4\eps}\right)^{2/3}
  =2^{-4/3}(\beta\delta^2)^{2/3}\eps^{1/3}.
\end{align*}
Adding the two terms proves \eqref{eq:regularized-value-asymptotic}.
\end{proof}

\begin{remark}
For the ellipsoidal boundary problem, \Cref{thm:universal-truncation} with $\nu=1$ shows that radius of order $\eta^{-1}$ is required to achieve objective residual $\eta$.  Objective tilting by $\tau$ produces radius $\tau^{-1/2}$ and residual $\tau^{1/2}$; see \Cref{thm:sensitivity}.  Quadratic regularization with coefficient $\eps$ selects radius $\eps^{-1/3}$ and residual $\eps^{1/3}$.  After eliminating $t$, and up to the additive constant $\eps\delta^2/2$, this is the case $q=2$ of \Cref{thm:universal-regularization} with the parameter rescaling $\lambda=2\eps$.  The three approximation mechanisms therefore select different scales.
\end{remark}

\section{Robust optimization interpretation}\label{sec:robust-extensions}

The ellipsoidal model is also the exact robust counterpart of a family of affine inequalities with joint coefficient uncertainty.  This representation transfers the SOCP solvability phase diagram directly to the robust model.

\subsection{Ellipsoidal coefficient uncertainty}\label{sec:robust}

Let
\[
  \mathbb B_{n+1}:=
  \{(\xi_0,\xi)\in\R\times\R^n:\xi_0^2+\norm{\xi}^2\leq1\}.
\]
For a decision $(t,z)$, consider the uncertain affine inequality
\begin{equation}\label{eq:robust-constraint}
  \delta \xi_0+\xi^{\trans}H^{1/2}z\leq t
  \qquad\text{for every }(\xi_0,\xi)\in\mathbb B_{n+1}.
\end{equation}
For each fixed uncertainty realization, the constraint is linear in $(t,z)$, with uncertainty in both the intercept and slope coefficients.

\begin{proposition}\label{prop:robust-counterpart}
The robust constraint \eqref{eq:robust-constraint} is equivalent to $(t,z)\in C_{\delta,H}$.
\end{proposition}

\begin{proof}
The left-hand side of the robust constraint has worst-case value
\[
  \sup_{\xi_0^2+\norm{\xi}^2\leq1}
  \bigl(\delta \xi_0+\xi^{\trans}H^{1/2}z\bigr)
  =\norm{(\delta,H^{1/2}z)}
  =\sqrt{\delta^2+z^{\trans}Hz}.
\]
Thus the robust constraint is equivalent to $(t,z)\in C_{\delta,H}$.
\end{proof}

Consider the robust problem
\begin{equation}\label{prob:robust}
\begin{aligned}
  \inf_{t\in\R,\ z\in\R^n}\quad&at+b^{\trans}z\\
  \text{subject to}\quad&
  \delta \xi_0+\xi^{\trans}H^{1/2}z\leq t
  \quad\forall (\xi_0,\xi)\in\mathbb B_{n+1}.
\end{aligned}
\end{equation}

Proposition~\ref{prop:robust-counterpart} and Theorem~\ref{thm:ellipsoidal-socp} give the following phase diagram.

\begin{corollary}\label{cor:robust-phase}
Let $\beta=\sqrt{b^{\trans}H^{-1}b}$.
\begin{enumerate}[label=\textup{(\roman*)}]
\item If $a>\beta$, problem \eqref{prob:robust} has the unique optimal solution \eqref{eq:ellipsoid-minimizer}.
\item If $a=\beta>0$, the robust problem is strictly feasible and has finite value zero, but no feasible point attains this value.
\item If $a<\beta$, the robust objective is unbounded below.
\item If $(a,b)=(0,0)$, every feasible point is optimal with value zero.
\end{enumerate}
\end{corollary}

\begin{remark}
For this ellipsoidal uncertainty model, nonattainment occurs only on the codimension-one boundary $a=\beta$; every strict margin $a>\beta$ yields a unique robust optimum.  At the boundary, however, feasibility, boundedness, strict conic feasibility, zero duality gap, and dual attainment all persist even though no primal robust optimizer exists.
\end{remark}

\section*{Acknowledgments}
The author gratefully acknowledges partial support from an AMS--Simons travel grant.

\end{document}